\documentclass{amsart}%
\usepackage{amssymb}
\usepackage{amsfonts}
\usepackage{amsmath}
\usepackage{graphicx}%
\providecommand{\U}[1]{\protect\rule{.1in}{.1in}}
\newtheorem{theorem}{Theorem}
\theoremstyle{plain}

\newtheorem{corollary}{Corollary}

\newtheorem{definition}{Definition}

\newtheorem{lemma}{Lemma}

\newtheorem{proposition}{Proposition}

\numberwithin{equation}{section}
\begin{document}
\title[An Inverse Function Theorem ]{An Inverse Function Theorem in Fr\'{e}chet spaces}
\author{Ivar Ekeland}
\address{Canada Research Chair in Mathematical Economics\\
University of British Columbia}
\email{ekeland@math.ubc.ca}
\urladdr{http://www.pims.math.ca/\symbol{126}ekeland}
\thanks{The author thanks Eric S\'{e}r\'{e}, Louis Nirenberg, Massimiliano Berti and
Philippe Bolle, who were the first to see this proof. Particular thanks go to
Philippe Bolle for spotting a significant mistake in an earlier version, and
to Eric S\'{e}r\'{e}, whose careful reading led to several improvements and
clarifications. Their contribution and friendship is gratefully acknowledged.}
\date{October 14, 2010; to appear, Annales de l'Institut Henri Poincar\'{e}, Analyse Non Lin\'{e}aire}

\begin{abstract}
I present an inverse function theorem for differentiable maps between
Fr\'{e}chet spaces which contains the classical theorem of Nash and Moser as a
particular case. In contrast to the latter, the proof does not rely on the
Newton iteration procedure, but on Lebesgue's dominated convergence theorem
and Ekeland's variational principle. As a consequence, the assumptions are
substantially weakened: the map $F$ to be inverted is not required to be
$C^{2}$, or even $C^{1}$,\ or even Fr\'{e}chet-differentiable.

\end{abstract}
\maketitle

\section{Introduction}

Recall that a Fr\'{e}chet space $X$ is \emph{graded} if its topology is
defined by an increasing sequence of norms $\left\Vert {}\right\Vert _{k}$,
$k\geq0$:%
\[
\forall x\in X,\ \ \ \left\Vert x\right\Vert _{k}\leq\left\Vert x\right\Vert
_{k+1}%
\]

Denote by $X_{k}$ the completion of $X$ for the norm $\left\Vert {}\right\Vert
_{k}$. It is a Banach space, and we have the following scheme:
\[
\longrightarrow X_{k+1}\longrightarrow_{i_{k}}X_{k}\longrightarrow_{i_{k-1}%
}X_{k-1}\longrightarrow...\longrightarrow X_{0}%
\]
where each identity map $i_{k}$ is injective and continuous, in fact
$\left\Vert i_{k}\right\Vert \leq1$. By definition, $X$ is a dense subspace of
$X_{k}$, we have $X=\cap_{k=0}^{\infty}X_{k}$, and $x^j\longrightarrow\bar{x}$
in $X$ if and only if $\left\Vert x^j-\bar{x}\right\Vert _{k}\longrightarrow0$
for every $k\geq0$.

Our main example will be $X=C^{\infty}\left(  \bar{\Omega},\mathbb{R}^{d}\right)  $, where
$ \bar{\Omega}\subset \mathbb{R}^{n}$ is compact, is the closure of its interior $\Omega$, and has
smooth boundary. It is well known that the topology of $C^{\infty}\left(
 \bar{\Omega},\mathbb{R}^{d}\right)  $ can be defined in two equivalent ways. On the one hand,
we can write $C^{\infty}\left(   \bar{\Omega},\mathbb{R}^{d}\right)  =\cap C^{k}\left(
 \bar{\Omega},\mathbb{R}^{d}\right)  $, where $C^{k}\left(   \bar{\Omega},\mathbb{R}^{d}\right)  $ is the
Banach space of all functions continuously differentiable up to order $k$,
endowed with the sup norm:%
\begin{equation}
\left\Vert x\right\Vert _{k}:=\max_{p_{1}+...+p_{n}\leq p}\max_{\omega
\in\bar{\Omega}}\left\vert \frac{\partial^{p_{1}+...+p_{n}}x}{\partial^{p_{1}}%
\omega_{1}...\partial^{p_{n}}\omega_{n}}\left(  \omega\right)  \right\vert
\label{08}%
\end{equation}

On the other, we can also write $C^{\infty}\left(  \bar{\Omega},\mathbb{R}^{d}\right)  =\cap
H^{k}\left(   \Omega,\mathbb{R}^{d}\right)  $, where $H^{k}\left(  \Omega,\mathbb{R}^{d}\right)$
is the Sobolev space consisting of all functions with square-integrable
derivatives, up to order $k$, endowed with the Hilbert space structure:%
\begin{equation}
\left\Vert x\right\Vert _{k}^{2}:=\sum_{p_{1}+...+p_{n}\leq p}\int_{\Omega
}\left\vert \frac{\partial^{p_{1}+...+p_{n}}x}{\partial^{p_{1}}\omega
_{1}...\partial^{p_{n}}\omega_{n}}\right\vert ^{2}d\omega\label{09}%
\end{equation}

We will prove an inverse function theorem between graded Fr\'{e}chet spaces.
Let us give a simple version in $C^{\infty}$; in the following statement,
either definition of the $k$-norms, (\ref{08}) or (\ref{09}), may be used:

\begin{theorem}
\label{thm0}Let $F$ be a map from a graded Fr\'{e}chet space $X=\cap_{k\geq
0}X_{k}$ into $C^{\infty}$. Assume that there are integers $d_{1}$ and $d_{2}%
$, and sequences $m_{k}>0,\ m_{k}^{\prime}>0,\ k\in\mathbb{N}$, such that, for
all $x$ is some neighbourhood of $0$ in $X$, we have:

\begin{enumerate}
\item $F\left(  0\right)  =0$

\item $F$ is continuous and G\^{a}teaux-differentiable, with derivative
$DF\left(  x\right)  $

\item For every $u\in X$, we have
\begin{equation}
\forall k\geq0,\ \ \left\Vert DF\left(  x\right)  u\right\Vert _{k}\leq
m_{k}\left\Vert u\right\Vert _{k+d_{1}} \label{01}%
\end{equation}

\item $DF\left(  x\right)  $ has a right-inverse $L\left(  x\right)  $:%
\[
\forall v\in C^{\infty},\ \ \ DF\left(  x\right)  L\left(  x\right)  v=v
\]

\item For every $v\in C^{\infty}$, we have:%
\begin{equation}
\forall k\geq0,\ \ \ \left\Vert L\left(  x\right)  v\right\Vert _{k}\leq
m_{k}^{\prime}\left\Vert v\right\Vert _{k+d_{2}}\label{03}%
\end{equation}

\end{enumerate}

Then for every $y\in C^{\infty}$ such that
\[
\left\Vert y\right\Vert _{k_{0}+d_{2}}<\frac{R}{m_{k_{0}}^{\prime}}%
\]
and every $m>m_{k_{0}}^{\prime}$ there is some $x\in X$ such that:%
\begin{align*}
\left\Vert x\right\Vert _{k_{0}}  &  <R\\
\left\Vert x\right\Vert _{k_{0}}  &  \leq m\left\Vert y\right\Vert
_{k_{0}+d_{2}}%
\end{align*}

and:%
\[
F\left(  x\right)  =y
\]

\end{theorem}

The full statement of our inverse function theorem and of its corollaries,
such as the implicit function theorem, will be given in the text (Theorems
\ref{Thm10} and \ref{cor3}). Note the main feature of our result: there is a
\emph{loss of derivatives} both for $DF\left(  x\right)  $, by condition
(\ref{01}), and for $L\left(  x\right)  $, by condition (\ref{03}).

Since the pioneering work of Andrei Kolmogorov and John Nash
(\cite{Kolmogorov}, \cite{Arnold}, \cite{Arnold2}, \cite{Arnold3},
\cite{Nash})\ , this loss of derivatives has been overcome by using the Newton
procedure:\ the equation $F\left(  x\right)  =y$ is solved by the iteration
scheme $x_{n+1}=x_{n}-L\left(  x_{n}\right)  F\left(  x_{n}\right)  $
(\cite{Schwartz}, \cite{Moser}, \cite{Moser2}; see \cite{Hamilton},
\cite{Alinhac} and \cite{BB} for more recent expositions). This method has two
drawbacks. The first one is that it requires the function $F$ to be $C^{2}$,
which is quite difficult to satisfy in infinite-dimensional situations. The
second is that it gives a set of admissible right-hand sides $y$ which is
unrealistically small: in practical situations, the equation $F\left(
x\right)  =y$ will continue to have a solution long after the Newton iteration
starting from $y$ ceases to converge.

Our method of is entirely different. It gives the solution of $F\left(
x\right)  =y$ directly by using Ekeland's variational principle (\cite{IE},
\cite{IE2}; see \cite{Ghoussoub} for later developments). Since the latter is
constructive, so is our proof, even if it does not rely on an iteration scheme
to solve the equation. To convey the idea of the method in a simple case, let
us now state and prove an inverse function theorem in Banach spaces:

\begin{theorem}
Let $X$ and $Y$ be Banach spaces. Let $F:X\rightarrow Y$ be continuous and
G\^{a}teaux-differentiable, with $F\left(  0\right)  =0$. Assume that the
derivative $DF\left(  x\right)  $ has a right-inverse $L\left(  x\right)  $,
uniformly bounded in a neighbourhood of $0$:%
\begin{align*}
&  \forall v\in Y,\ \ DF\left(  x\right)  L\left(  x\right)  v=v\ \\
&  \left\Vert x\right\Vert \leq R\Longrightarrow\left\Vert L\left(  x\right)
\right\Vert \leq m
\end{align*}

Then, for every $\bar{y}$ such that:%
\[
\left\Vert \bar{y}\right\Vert <\frac{R}{m}%
\]
and every $\mu>m$ there is some $\bar{x}$ such that:%
\begin{align*}
\left\Vert \bar{x}\right\Vert  &  <R\\
\left\Vert \bar{x}\right\Vert  &  \leq\mu\left\Vert \bar{y}\right\Vert
\end{align*}
and:%
\[
F\left(  \bar{x}\right)  =\bar{y}%
\]

\end{theorem}

\begin{proof}
Consider the function $f:X\rightarrow R$ defined by:%
\[
f\left(  x\right)  =\left\Vert F\left(  x\right)  -\bar{y}\right\Vert
\]

It is continuous and bounded from below, so that we can apply Ekeland's
variational principle. For every $r>0$, we can find a point $\bar{x}$ such
that:
\begin{align*}
f\left(  \bar{x}\right)   &  \leq f\left(  0\right) \\
\left\Vert \bar{x}\right\Vert  &  \leq r\\
\forall x,\ \ \ f\left(  x\right)   &  \geq f\left(  \bar{x}\right)
-\frac{f\left(  0\right)  }{r}\left\Vert x-\bar{x}\right\Vert
\end{align*}

Note that $f\left(  0\right)  =\left\Vert \bar{y}\right\Vert $. Take
$r=\mu\left\Vert \bar{y}\right\Vert $, so that:%
\begin{align}
f\left(  \bar{x}\right)   &  \leq f\left(  0\right)  \nonumber\\
\left\Vert \bar{x}\right\Vert  &  \leq\mu\left\Vert \bar{y}\right\Vert
\nonumber\\
\forall x,\ \ \ f\left(  x\right)   &  \geq f\left(  \bar{x}\right)  -\frac
{1}{\mu}\left\Vert x-\bar{x}\right\Vert \label{s}%
\end{align}

The point $\bar{x}$ satisfies the inequality $\left\Vert \bar{x}\right\Vert
\leq\mu\left\Vert \bar{y}\right\Vert $. Since $m<R\left\Vert \bar
{y}\right\Vert ^{-1}$, we can assume without loss of generality that
$\mu<R\left\Vert \bar{y}\right\Vert ^{-1}$, so $\left\Vert \bar{x}\right\Vert
<R$. It remains to prove that $F\left(  \bar{x}\right)  =\bar{y}$.

We argue by contradiction. Assume not, so $F\left(  \bar{x}\right)  \neq
\bar{y}$. Then, write the inequality (\ref{s}) with $x=\bar{x}+tu$. We get:%
\begin{equation}
\forall t>0,\ \forall u\in X,\ \ \ \frac{\left\Vert F\left(  \ \bar
{x}+tu\right)  -\bar{y}\right\Vert -\left\Vert F\left(  \bar{x}\right)
-\bar{y}\right\Vert }{t}\geq-\frac{1}{\mu}\left\Vert u\right\Vert \label{q20}%
\end{equation}

As we shall see later on (Lemma \ref{l2}), the function $t\rightarrow
\left\Vert F\left(  \ \bar{x}+tu\right)  -\bar{y}\right\Vert $ is
right-differentiable at $t=0$, and its derivative is given by:%
\[
\lim_{\substack{t\rightarrow0 \\t>0}}\frac{\left\Vert F\left(  \ \bar
{x}+tu\right)  -\bar{y}\right\Vert -\left\Vert F\left(  \bar{x}\right)
-\bar{y}\right\Vert }{t}=\ <y^{\ast},DF\left(  \bar{x}\right)  u>
\]
for some $y^{\ast}\in Y^{\ast}$ with $\left\Vert y^{\ast}\right\Vert ^{\ast
}=1$ and $<y^{\ast},F\left(  \bar{x}\right)  -\bar{y}>\ =\left\Vert F\left(
\bar{x}\right)  -\bar{y}\right\Vert $. In the particular case when $Y$ is a
Hilbert space, we have
\[
y^{\ast}=\frac{F\left(  \bar{x}\right)  -\bar{y}}{\left\Vert F\left(  \bar
{x}\right)  -\bar{y}\right\Vert }%
\]
Letting $t\rightarrow0$ in (\ref{q20}), we get:%
\[
\forall u,\ \ <y^{\ast},DF\left(  \bar{x}\right)  u>\ \geq-\frac{1}{\mu
}\left\Vert u\right\Vert
\]

We now take $u=-L\left(  \bar{x}\right)  \left(  F\left(  \bar{x}\right)
-\bar{y}\right)  $, so that\ $DF\left(  \bar{x}\right)  u=-\left(  F\left(
\bar{x}\right)  -\bar{y}\right)  $. The preceding inequality yields:%
\[
\left\Vert F\left(  \bar{x}\right)  -\bar{y}\right\Vert \leq\frac{1}{\mu
}\left\Vert u\right\Vert \leq\frac{m}{\mu}\left\Vert F\left(  \bar{x}\right)
-\bar{y}\right\Vert
\]
which is a contradiction since $\mu>m$.
\end{proof}

The reader will have noted that this is much stronger than the usual inverse
function theorem in Banach spaces: we do not require that $F$ be
Fr\'{e}chet-differentiable, nor that the derivative $DF\left(  x\right)  $ or
its inverse $L\left(  x\right)  $ depend continuously on $x$. All that is
required is an upper bound on $L\left(  x\right)  $. Note that it is very
doubtful that, with such weak assumptions, the usual Euler or Newton iteration
schemes would converge.

Our inverse function theorem will extend this idea to Fr\'{e}chet spaces.
Ekeland's variational principle holds for any complete metric space, hence on
Fr\'{e}chet spaces. Our first result, Theorem \ref{thm2}, depends on the
choice of a\ non-negative sequence $\beta_{k}$ wich converges rapidly to zero.
An appropriate choice of $\beta_{k}$ will lead to the facts that $F$ is a
local surjection (Corollary \ref{cor1}), that the (multi-valued) local inverse
$F^{-1}$ satisfies a Lipschitz condition (Corollary \ref{cor5}), and that the
result holds also with finite regularity (if $\bar{y}$ does not belong to
$Y_{k}$ for every $k$, but only to $Y_{k_{0}+d_{2}}$, then we can still solve
$F\left(  \bar{x}\right)  =\bar{y}$ with $\bar{x}\in X_{k_{0}}$). These
results are gathered together in Theorem \ref{Thm10}, which is our final
inverse function theorem. As usual, an implicit function theorem can be
derived; it is given in Theorem \ref{cor3}.

The structure of the paper is as follows. Section 2 introduces the basic
definitions. There will be no requirement on $X$, while $Y$ will be asked to
belong to a special class of graded Fr\'{e}chet spaces.\ This class is much
larger than the one used in the Nash-Moser literature:\ it is not required
that $X$ or $Y$ be tame in the sense of Hamilton \cite{Hamilton} or admit
smoothing operators. Section 3 states Theorem \ref{thm2} and derives the other
results. The proof of Theorem \ref{thm2} is given in Section 4.

Before we proceed, it will be convenient to recall some well-known facts
about differentiability in Banach spaces. Let $X$ be a Banach space.
Let $\mathcal{U}$ be some open subset of $X$, and
let $F$ be a map from $\mathcal{U}$ into some Banach space $Y$.

\begin{definition}
$F$ is \emph{G\^{a}teaux-differentiable} at $x\in \mathcal{U}$ if there exists some linear
continuous map from $X$ to $Y$, denoted by $DF\left(  x\right)  $, such that:%
\[
\forall u\in X,\ \ \ \lim_{t\longrightarrow0}\left\Vert \frac{F\left(
x+tu\right)  -F\left(  x\right)  }{t}-DF\left(  x\right)  u\right\Vert =0
\]

\end{definition}

\begin{definition}
$F$ is \emph{Fr\'{e}chet-differentiable} at $x\in \mathcal{U}$ if it is
G\^{a}teaux-differentiable and:%
\[
\lim_{u\rightarrow0}\frac{F\left(  x+u\right)  -F\left(  x\right)  -DF\left(
x\right)  u}{\left\Vert u\right\Vert }=0
\]

\end{definition}

Fr\'{e}chet-differentiability implies G\^{a}teaux-differentiability and
continuity, but G\^{a}teaux-differentiability does not even imply continuity.

If $Y$ is a Banach space, then the norm $y\longrightarrow\left\Vert
y\right\Vert $ is convex and continuous, so that it has a non-empty
subdifferential $N\left(  y\right)  \subset Y^{\ast}$ at every point $y:$%
\[
y^{\ast}\in N\left(  y\right)  \Longleftrightarrow\forall z\in Y,\ \left\Vert
y+z\right\Vert \geq\left\Vert y\right\Vert +\left\langle y^{\ast
},z\right\rangle
\]

When $y\neq 0$ we have the alternative characterization:%
\[
y^{\ast}\in N\left(  y\right)  \Longleftrightarrow\left\Vert y^{\ast
}\right\Vert ^{\ast}=1\text{ and }\left\langle y^{\ast},y\right\rangle
=\left\Vert y\right\Vert
\]

$N\left(  y\right)  $ is convex and compact in the $\sigma\left(  X^{\ast
},X\right)  $-topology. It is a singleton if and only if the norm is
G\^{a}teaux-differentiable at $y$. Its unique element then is the
G\^{a}teaux-derivative of $F$ at $y$:%
\[
N\left(  y\right)  =\left\{  y^{\ast}\right\}  \Longleftrightarrow DF\left(
x\right)  =y^{\ast}%
\]

If $N\left(  y\right)  $ contains several elements, the norm is not
G\^{a}teaux-differentiable at $y$, and the preceding relation is then replaced
by the following classical result (see for instance Theorem 11 p. 34 in
\cite{Rockafellar}):

\begin{proposition}
\label{p1}Take $y$ and $z$ in $Y$.
Then there is some $y^{\ast}\in
N\left(  y\right)  $ (depending possibly on $z$) such that:%
\[
\lim_{\substack{t\rightarrow0 \\t>0}}\frac{\left\Vert y+tz\right\Vert
-\left\Vert y\right\Vert }{t}=\ <y^{\ast},z>
\]

\end{proposition}

The following result will be used repeatedly:

\begin{lemma}
\label{l2}Assume $F:X\longrightarrow Y$ is G\^{a}teaux-differentiable. Take
$x$ and $\xi$ in $X$, and define a function $f:\left[  0,\ 1\right]
\rightarrow R$ by:%
\[
f\left(  t\right)  :=\left\Vert F\left(  x+t\xi\right)  \right\Vert ,\ \ 0\leq
t\leq1
\]

Then $f$ has a right-derivative everywhere, and there is some $y^{\ast}\in
N\left(  F\left(  x+t\xi\right)  \right)  $ such that:%
\[
\lim_{\substack{h\rightarrow0 \\h>0}}\frac{f\left(  t+h\right)  -f\left(
t\right)  }{h}=\ <y^{\ast},DF\left(  x+t\xi\right)  \xi>
\]

\end{lemma}

\begin{proof}
We have:
\begin{align}
h^{-1}\left[  f\left(  t+h\right)  -f\left(  t\right)  \right]   &
=h^{-1}\left[  \left\Vert F\left(  x+\left(  t+h\right)  \xi\right)
\right\Vert -\left\Vert F\left(  x+t\xi\right)  \right\Vert \right]
\nonumber\\
&  =h^{-1}\left[  \left\Vert F\left(  x+t\xi\right)  +hz\left(  h\right)
\right\Vert -\left\Vert F\left(  x+t\xi\right)  \right\Vert \right]
\label{q50}%
\end{align}
with:%
\[
z\left(  h\right)  =\frac{F\left(  x+\left(  t+h\right)  \xi\right)  -F\left(
x+t\xi\right)  }{h}%
\]

Since $F$ is G\^{a}teaux-differentiable, we have:%
\[
\lim_{h\rightarrow0}z\left(  h\right)  =DF\left(  x+t\xi\right)  \xi:=z\left(
0\right)
\]

By the triangle inequality, we have:%
\[
\bigl\vert \left\Vert F\left(  x+t\xi\right)  +hz\left(  h\right)  \right\Vert
-\left\Vert F\left(  x+t\xi\right)  +hz\left(  0\right)  \right\Vert
\bigr\vert \leq h\left\Vert z\left(  h\right)  -z\left(  0\right)
\right\Vert
\]

Writing this into (\ref{q50}) and using Proposition \ref{p1}, we find;%
\begin{align*}
\lim_{\substack{h\rightarrow0 \\h>0}}\frac{\left[  f\left(  t+h\right)
-f\left(  t\right)  \right]  }{h}  &  =\lim_{\substack{h\rightarrow0
\\h>0}}\frac{\left\Vert F\left(  x+t\xi\right)  +hz\left(  0\right)
\right\Vert -\left\Vert F\left(  x+t\xi\right)  \right\Vert }{h}\\
&  =\ <y^{\ast},DF\left(  x+t\xi\right)  \xi>
\end{align*}
for some $y^{\ast}\in N\left(  F\left(  x+t\xi\right)  \right)  $. This is the
desired result.
\end{proof}

One last word. Throughout the paper, we shall use the following:

\begin{definition}
A sequence $\alpha_{k}$ has \emph{unbounded support} if $\sup\left\{
k\ |\ \alpha_{k}\neq0\right\}  =\infty\,.$
\end{definition}

\section{The Inverse Function Theorem}

Let $X=\cap_{k\geq0}X_{k}$ be a graded Fr\'{e}chet space. The following result
is classical:

\begin{proposition}
\label{p2}Let $\alpha_{k}\geq0$ be a sequence with unbounded support such that
$\sum\alpha_{k}<\infty$. Let $r>0$ be a positive number. Then the topology of
$X$ is induced by the distance:%
\begin{equation}
d\left(  x,y\right)  :=\sum_{k}\alpha_{k}\min\left\{  r,\left\Vert
x-y\right\Vert _{k}\right\}  \label{x1}%
\end{equation}
and $x_{n}$ is a Cauchy sequence for $d$ if and only if it is a Cauchy
sequence for all the $k$-norms. It follows that $\left(  X,d\right)  $ is a
complete metric space.
\end{proposition}

The main analytical difficulty with Fr\'{e}chet spaces is that, given $x\in X
$, there is no information on the sequence $k\rightarrow\left\Vert
x\right\Vert _{k}$, except that it is positive and increasing. For instance,
it can grow arbitrarily fast. So we single out some elements $x$ such that
$\left\Vert x\right\Vert _{k}$ has at most exponential growth in $k\,.$

\begin{definition}
A point $x\in X$ is \emph{controlled} if there is a constant $c_{0}\left(
x\right)  $ such that:%
\begin{equation}
\left\Vert x\right\Vert _{k}\leq c_{0}\left(  x\right)  ^{k} \label{111}%
\end{equation}

\end{definition}

\begin{definition}
A graded Fr\'{e}chet space is \emph{standard} if, for every $x\in X$, there is
a constant $c_{3}\left(  x\right)  $ and a sequence $x_{n}$ such that:%
\begin{align}
\ \ \forall k\ \ \ \lim_{n\rightarrow\infty}\left\Vert x_{n}-x\right\Vert
_{k}  &  =0\label{50}\\
\forall n,\ \ \left\Vert x_{n}\right\Vert _{k}  &  \leq c_{3}\left(  x\right)
\left\Vert x\right\Vert _{k} \label{51}%
\end{align}
and each $x_{n}$ is controlled:%
\begin{equation}
\left\Vert x_{n}\right\Vert _{k}\leq c_{0}\left(  x_{n}\right)  ^{k}
\label{510}%
\end{equation}

\end{definition}

\begin{proposition}
The graded Fr\'{e}chet spaces $C^{\infty}\left( \bar{\Omega},\mathbb{R}^{d}\right)  =\cap
C^{k}\left( \bar{\Omega},\mathbb{R}^{d}\right)  $ and\break $C^{\infty}\left(  \bar{\Omega},\mathbb{R}^{d}\right)
=\cap H^{k}\left(  \Omega,\mathbb{R}^{d}\right)  $ are both standard.
\end{proposition}

\begin{proof}
In fact, much more is true. It can be proved (see \cite{Hamilton},
\cite{Alinhac} Proposition II.A.1.6) that both admit a family of smoothing
operators $S_{n}:X\rightarrow X$ satisfying:

\begin{enumerate}
\item $\left\Vert S_{n}x-x\right\Vert _{k}\rightarrow0$ when $n\rightarrow
\infty$

\item $\left\Vert S_{n}x\right\Vert _{k+d}\leq c_{1}n^{d}\left\Vert
x\right\Vert _{k}\ \ \ \forall d\geq0,\ \forall k\geq0,\ \forall n\geq0$

\item $\left\Vert \left(  I-S_{n}\right)  x\right\Vert _{k}\leq c_{2}%
n^{-d}\left\Vert x\right\Vert _{k+d}\ \ \ \forall d\geq0,\ \forall
k\geq0,\ \forall n\geq0$
\end{enumerate}

where $c_{1}$ and $c_{2}$ are positive constants. For every $x\in X$,
condition (2) with $k=0$ implies that:%
\[
\left\Vert S_{n}x\right\Vert _{d}\leq c_{1}n^{d}\left\Vert x\right\Vert _{0}%
\]

So (\ref{111}) is satisfied and the point $x_{n}=S_{n}x$ is controlled.
Condition (\ref{50}) follows from (1) and condition (\ref{51}) follows from
(2) with $d=0$.
\end{proof}

Now consider two graded Fr\'{e}chet spaces $X=\cap_{k\geq0}X_{k}$ and
$Y=\cap_{k\geq0}Y_{k}$. We are given a map $F:X\longrightarrow Y$, a number
$\ R\in(0,\infty]$ and an integer $k_{0}\geq0$. We consider the ball:
\[
B_{X}\left(  k_{0},R\right)  :=\left\{  x\in X\ |\ \left\Vert x\right\Vert
_{k_{0}}<R\right\}
\]

Note that $R=+\infty$ is allowed; in that case, $B_{X}\left(  k_{0},R\right)
=X$.

\begin{theorem}
\label{thm2}Assume $Y$ is standard. Let there be given two integers
$d_{1},\ d_{2}$ and two non-decreasing sequences $m_{k}>0,\ m_{k}^{\prime}>0$.
Assume that, on $B_{X}\left(  k_{0},R\right)  $, the map $F$ satisfies the
following conditions:

\begin{enumerate}
\item $F\left(  0\right)  =0$

\item $F$ is continuous, and G\^{a}teaux-differentiable

\item For every $u\in X$ there is a number $c_{1}\left(  u\right)  >0$ for
which:
\[
\forall k\in\mathbb{N},\ \ \ \left\Vert DF\left(  x\right)  u\right\Vert
_{k}\leq c_{1}\left(  u\right)  \left(  m_{k}\left\Vert u\right\Vert
_{k+d_{1}}+\left\Vert F\left(  x\right)  \right\Vert _{k}\right)
\]

\item There exists a linear map $L\left(  x\right)  :Y\longrightarrow X$ such that
$DF\left(  x\right)  L\left(  x\right)  =I_{Y}$: $\ $%
\[
\forall v\in Y,\ DF\left(  x\right)  L\left(  x\right)  v=v
\]

\item For every $v\in Y$:
\[
\forall k\in\mathbb{N},\ \ \ \left\Vert L\left(  x\right)  v\right\Vert
_{k}\leq m_{k}^{\prime}\left\Vert v\right\Vert _{k+d_{2}}%
\]

\end{enumerate}

Let $\beta_{k}\geq0$ be a sequence with unbounded support satisfying:%
\begin{equation}
\forall n\in\mathbb{N},\ \ \ \sum_{k=0}^{\infty}\beta_{k}m_{k}m_{k+d_{1}%
}^{\prime}n^{k}<\infty\ \label{y20}%
\end{equation}
Then, for every $\bar{y}\in Y$ such that:
\begin{equation}
\sum_{k=0}^{\infty}\beta_{k}\left\Vert \bar{y}\right\Vert _{k}<\frac
{\beta_{k_{0}+d_{2}}}{m_{k_{0}}^{\prime}}R\label{y21}%
\end{equation}
there exists a point $\bar{x}\ $such that:%
\begin{align}
F\left(  \bar{x}\right)   &  =\bar{y}\label{y22}\\
\left\Vert \bar{x}\right\Vert _{k_{0}} &  <R\label{t220}%
\end{align}

\end{theorem}

The proof is given in the next section.\ We will now derive the consequences.
Before we do that, note that Condition 3 is equivalent to the following,
seemingly more general, one:%
\begin{equation}
\forall k,\ \ \ \left\Vert DF\left(  x\right)  u\right\Vert _{k}\leq
c_{1}^{\prime}\left(  u\right)  \left(  m_{k}\left\Vert u\right\Vert
_{k+d_{1}}+\left\Vert F\left(  x\right)  \right\Vert _{k}+1\right)
\label{q24}%
\end{equation}
Indeed, Condition 3 clearly implies (\ref{q24}) with\ $c_{1}^{\prime}\left(
u\right)  =c_{1}\left(  u\right)  $. Conversely, if (\ref{q24}) holds,  for $u\neq 0$ we set:%
\[
c_{1}\left(  u\right)  =c_{1}^{\prime}\left(  u\right)  \left(  1+\frac
{1}{m_{0}\left\Vert u\right\Vert_{0}}\right)
\]
and then:%
\[
c_{1} \left(  u\right)  \left(  m_{k}\left\Vert u\right\Vert
_{k+d_{1}}+\left\Vert F\left(  x\right)  \right\Vert _{k}\right)  \geq
c_{1}^{\prime}\left(  u\right)  \left(  m_{k}\left\Vert u\right\Vert _{k+d_{1}%
}+\left\Vert F\left(  x\right)  \right\Vert _{k}+1\right)
\]
so Condition 3 holds as well.

\begin{corollary}
[Local Surjection]\label{cor1}Let $X=\cap_{k\geq0}X_{k}$ and $Y=\cap_{k\geq
0}Y_{k}$ be graded Fr\'{e}chet spaces, with $Y$ standard, and let
$F:X\rightarrow Y$ satisfy conditions 1 to 5. Suppose:%
\begin{equation}
\left\Vert y\right\Vert _{k_{0}+d_{2}}<\frac{R}{m_{k_{0}}^{\prime}}
\label{ek12}%
\end{equation}

Then, for every $\mu>m_{k_{0}}^{\prime}$, there is some $x\in B_{X}\left(
k_{0},R\right)  $ such that:%
\begin{equation}
\left\Vert x\right\Vert _{k_{0}} \leq \mu\left\Vert y\right\Vert _{k_{0}+d_{2}}
\label{ek13}%
\end{equation}
and:%
\[
F\left(  x\right)  =y
\]

\end{corollary}

\begin{proof}
Given $\mu>m_{k_{0}}^{\prime}$, take $R^{\prime}<R$ such that $m_{k_{0}%
}^{\prime}\left\Vert y\right\Vert _{k_{0}+d_{2}}<R^{\prime}<\mu\left\Vert
y\right\Vert _{k_{0}+d_{2}}$ and choose $h>0$ so small that:%
\begin{equation}
\left\Vert y\right\Vert _{k_{0}+d_{2}}+h\sum_{k=0}^{\infty}k^{-k}%
<\frac{R^{\prime}}{m_{k_{0}}^{\prime}} \label{ek100}%
\end{equation}

Define a sequence $\beta_{k}$ by:%
\begin{equation}
\beta_{k}=\left\{
\begin{array}
[c]{c}%
0\text{ for }k<k_{0}+d_{2}\\
1\text{ for }k=k_{0}+d_{2}\\
hk^{-k}\left[  \max\left\{  \left\Vert y\right\Vert _{k},m_{k}m_{k+d_{1}%
}^{\prime}\right\}  \right]  ^{-1}\text{ for }k>k_{0}+d_{2}%
\end{array}
\right.  \label{ek25}%
\end{equation}

Then (\ref{y20}) is satisfied. On the other hand:
\[
\sum_{k=0}^{\infty}\beta_{k}\left\Vert y\right\Vert _{k}\leq\left\Vert
y\right\Vert _{k_{0}+d_{2}}+h\sum k^{-k}%
\]
and using (\ref{ek100}) we find that:%
\begin{equation}
\sum_{k=0}^{\infty}\beta_{k}\left\Vert y\right\Vert _{k}<\frac{1}{m_{k_{0}%
}^{\prime}}R^{\prime}<\infty\label{ek26}%
\end{equation}

We now apply Theorem \ref{thm2} with $R^{\prime}$ instead of $R$, and we find
some $x\in X$ with $F\left(  x\right)  =y$ and $\left\Vert x\right\Vert
_{k_{0}}\leq R^{\prime}$. The result follows.
\end{proof}

Consider the two balls:%
\begin{align*}
B_{X}\left(  k_{0},R\right)   &  =\left\{  x\ |\ \left\Vert x\right\Vert
_{k_{0}}<R\right\} \\
B_{Y}\left(  k_{0}+d_{2},\frac{R}{m_{k_{0}}^{\prime}}\right)   &  =\left\{
y\ |\ \left\Vert y\right\Vert _{k_{0}+d_{2}}<\frac{R}{m_{k_{0}}^{\prime}%
}\right\}
\end{align*}

The preceding Corollary tells us that $F$ maps the first ball onto the
second. The inverse map
\[
F^{-1}:B_{Y}\left(  k_{0}+d_{2},\frac{R}{m_{k_{0}}^{\prime}}\right)
\mathcal{\rightarrow}B_{X}\left(  k_{0},R\right)
\]
is multivalued:%
\begin{equation}
F^{-1}\left(  y\right)  =\left\{  x\in B_{X}\left(  k_{0},R\right)
\ |\ F\left(  x\right)  =y\right\}  \label{ek50}%
\end{equation}
and has non-empty values: $F^{-1}\left(  y\right)  \neq\varnothing$ for every
$y$. The following result shows that it satisfies a Lipschitz condition.

\begin{corollary}
[Lipschitz inverse]\label{cor5}$\ $For every $y_{0}$ and $y_{1}$ in
$B_{Y}\left(  k_{0}+d_{2},R\left(  m_{k_{0}}^{\prime}\right)  ^{-1}\right)  $,
every $x_{0}\in F^{-1}\left(  y_{0}\right)  $, and every $\mu>m_{k_{0}%
}^{\prime}$, we have:%
\begin{align*}
\inf\left\{  \left\Vert x_{0}-x_{1}\right\Vert _{k_{0}}\ |\ \ x_{1}\in
F^{-1}\left(  y_{1}\right)  \right\}   &  =\inf\left\{  \left\Vert x_{0}%
-x_{1}\right\Vert _{k_{0}}\ |\ F\left(  x_{1}\right)  =y_{1}\right\} \\
&  \leq\mu\left\Vert y_{0}-y_{1}\right\Vert _{k_{0}+d_{2}}%
\end{align*}

\end{corollary}

\begin{proof}
Take some $R^{\prime}<R$ with $\max\left\{  \left\Vert y_{0}\right\Vert
_{k_{0}+d_{2}},\left\Vert y_{1}\right\Vert _{k_{0}+d_{2}}\right\}  <R^{\prime
}\left(  m_{k_{0}}^{\prime}\right)  ^{-1}$ and consider the line segment
$y_{t}=y_{0}+t\left(  y_{1}-y_{0}\right)  $, $0\leq t\leq1$, joining $y_{0}$
to $y_{1}$. We have $\left\Vert y_{t}\right\Vert <R^{\prime}\left(  m_{k_{0}%
}^{\prime}\right)  ^{-1}$ for every $t$, so that, by Corollary \ref{cor1},
there exists some $x_{t}\in F^{-1}\left(  y_{t}\right)  $ with $\left\Vert
x_{t}\right\Vert <R^{\prime}$. The function $F_{t}\left(  x\right)  =F\left(
x+x_{t}\right)  -y_{t}$ then satisfies Conditions 1 to 5 with $R$ replaced by
$\rho=R-R^{\prime}$.

Pick some $x_{0}\in F^{-1}\left(  y_{0}\right)  $. By Corollary\ \ref{cor1}
applied to $F_{0}$ we find that, for every $y$ such that $\left\Vert
y-y_{0}\right\Vert _{k_{0}+d_{2}}\leq\rho\left(  m_{k_{0}}^{\prime}\right)
^{-1}$ we have some $x\in F^{-1}\left(  y\right)  $ with:
\[
\,\left\Vert x_{0}-x\right\Vert _{k_{0}}\leq\mu\left\Vert y_{0}-y\right\Vert
_{k_{0}}%
\]

We can connect $y_{0}$ and $y_{1}$ by a finite chain $y_{0}^{\prime}%
=y_{0},\ y_{2}^{\prime},...y_{N}^{\prime}=y_{1}$ of aligned points, such that
the distance between $y_{n}^{\prime}$ and $y_{n+1}^{\prime}$ is always less
than $\rho\left(  m_{k_{0}}^{\prime}\right)  ^{-1}$, and for each
$y_{n}^{\prime}$ choose some $x_{n}\in F^{-1}\left(  y_{n}^{\prime}\right)  $
such that:
\[
\left\Vert x_{n}-x_{n+1}\right\Vert _{k_{0}}\leq\mu\left\Vert y_{n}^{\prime
}-y_{n+1}^{\prime}\right\Vert _{k_{0}+d_{2}}%
\]

Summing up:%
\[
\left\Vert x_{0}-x_{1}\right\Vert _{k_{0}}\leq\mu\sum_{n=0}^{N}\left\Vert
y_{n}^{\prime}-y_{n+1}^{\prime}\right\Vert _{k_{0}+d_{2}}=\mu\left\Vert
y_{1}-y_{0}\right\Vert _{k_{0}+d_{2}}%
\]

\end{proof}

Note that we are not claiming that the multivalued map $F^{-1}$ has a
Lipschitz section over $B_{Y}\left(  k_{0}+d_{2},R\left(  m_{k_{0}}^{\prime
}\right)  ^{-1}\right)  $, or even a continous one.

As a consequence of Corollary \ref{cor5}, we can solve the equation $F\left(
x\right)  =y$ when the right-hand side no longer is in $Y$, but in some of the
$Y_{k}$, with $k\geq k_{0}+d_{2}\,.$

\begin{corollary}
[Finite regularity]\label{cor6} Suppose $F$ extends to a continuous map
$\bar{F}:X_{k_{0}}\rightarrow Y_{k_{0}-d_{1}}$. Take some $y\in Y_{k_{0}%
+d_{2}}$ with $\left\Vert y\right\Vert _{k_{0}+d_{2}}<R\left(  m_{k_{0}%
}^{\prime}\right)  ^{-1}$. Then there is some $x\in X_{k_{0}}$ such that
$\left\Vert x\right\Vert _{k_{0}}<R$ and $\bar{F}\left(  x\right)  =y\,.$
\end{corollary}

\begin{proof}
Let $y_{n}\in Y$ be such that $\left\Vert y_{n}-y\right\Vert _{k_{0}+d}%
\leq2^{-n}$. By Corollary \ref{cor5}, we can find a sequence $x_{n}\in X$ such
that $F\left(  x_{n}\right)  =y_{n}$ and
\[
\left\Vert x_{n}-x_{n+1}\right\Vert _{k_{0}}\leq\mu\left\Vert y_{n}%
-y_{n+1}\right\Vert _{k_{0}+d}\leq\mu\,2^{-n}%
\]

So $\left\Vert x_{n}-x_{p}\right\Vert _{k_{0}}\leq\mu2^{-n+1}$ for $p>n$, and
the sequence $x_{n}$ is Cauchy in $X_{k_{0}}$. It follows that $x_{n}$
converges to some $x\in X_{k_{0}}$, with $\left\Vert x\right\Vert _{k_{0}}<R$,
and we get $F\left(  x\right)  =y$ by continuity.
\end{proof}

Let us sum up our results in a single statement:

\begin{theorem}
[Inverse Function Theorem]\label{Thm10} Let $X=\cap_{k\geq0}X_{k}$ and
$Y=\cap_{k\geq0}Y_{k}$ be graded Fr\'{e}chet spaces, with $Y$ standard, and
let $F$ be a map from $X$ to $Y$. Assume there exist some integer $k_{0}$,
some $R>0$ (possibly equal to $+\infty$), integers $d_{1},\ d_{2}$, and
non-decreasing sequences $m_{k}>0,\ m_{k}^{\prime}>0$ such that, for
$\left\Vert x\right\Vert _{k_{0}}<R$, we have:

\begin{enumerate}
\item $F\left(  0\right)  =0$

\item $F$ is continuous, and G\^{a}teaux-differentiable with derivative
$DF\left(  x\right)  $

\item For every $u\in X$ there is a number $c_{1}\left(  u\right)  $ such
that:
\[
\forall k,\text{\ \ }\left\Vert DF\left(  x\right)  u\right\Vert _{k}\leq
c_{1}\left(  u\right)  \left(  m_{k}\left\Vert u\right\Vert _{k+d_{1}%
}+\left\Vert F\left(  x\right)  \right\Vert _{k}\right)
\]

\item There exists a linear map $L\left(  x\right)  :Y\longrightarrow X$ such that
$\ DF\left(  x\right)  L\left(  x\right)  =I_{Y}$%
\[
\forall u\in X,\ \ DF\left(  x\right)  L\left(  x\right)  v=v
\]

\item For every $v\in Y$, we have:%
\[
\forall k\in\mathbb{N},\ \ \ \left\Vert L\left(  x\right)  v\right\Vert
_{k}\leq m_{k}^{\prime}\left\Vert v\right\Vert _{k+d_{2}}%
\]

\end{enumerate}

Then $F$ maps the ball $\bigl\{\left\Vert x\right\Vert _{k_{0}}<R\bigr\}$ in $X$ onto the
ball $\bigl\{\left\Vert y\right\Vert _{k_{0}+d_{2}}<R\left(  m_{k_{0}}^{\prime
}\right)  ^{-1}\bigr\}$ in $Y$, and for every $\mu>m_{k_{0}}^{\prime}$the inverse
$F^{-1}$ satisfies a Lipschitz condition:%
\[
\forall x_{1}\in F^{-1}\left(  y_{1}\right)  ,\ \ \inf\left\{  \left\Vert
x_{1}-x_{2}\right\Vert _{k_{0}}\ |\ x_{2}\in F^{-1}\left(  y_{2}\right)
\right\}  \leq \mu\left\Vert y_{1}-y_{2}\right\Vert _{k_{0}+d_{2}}%
\]

If $F$ extends to a continuous map $\bar{F}:X_{k_{0}}\rightarrow
Y_{k_{0}-d_{1}}$, then $\bar{F}$ maps the ball $\bigl\{\left\Vert x\right\Vert
_{k_{0}}<R\bigr\}$ in $X_{k_{0}}$ onto the ball $\bigl\{\left\Vert y\right\Vert
_{k_{0}+d_{2}}<R\left(  m_{k_{0}}^{\prime}\right)  ^{-1}\bigr\}$ in $Y_{k_{0}+d_{2}}%
$, and the inverse $\bar{F}^{-1}$ satisfies the same Lipschitz condition.
\end{theorem}

We conclude by rephrasing Theorem \ref{Thm10} as an implicit function theorem.

\begin{theorem}
\label{cor3} Let $X=\cap_{k\geq0}X_{k}$ and $Y=\cap_{k\geq0}Y_{k}$ be graded
Fr\'{e}chet spaces, with $Y$ standard, and let $F\left(  \varepsilon,x\right)
=F_{0}\left(  x\right)  +\varepsilon F_{1}\left(  x\right)  $ be a map from
$X\times \mathbb{R}$ to $Y$. Assume there exist integers $k_{0},\ d_{1},\ d_{2}$,
sequences $m_{k}>0,\ m_{k}^{\prime}>0$, and numbers $R>0$\thinspace$,$
$\varepsilon_{0}>0$ such that, for every $\left(  x,\varepsilon\right)  $ such
that $\left\Vert x\right\Vert _{k_{0}}\leq R$ and $\left\vert \varepsilon
\right\vert <\varepsilon_{0}$, we have:

\begin{enumerate}
\item $F_{0}\left(  0\right) =0$ and $F_{1}\left(  0\right)  \neq0$

\item $F_{0}$ and $F_{1}$ are continuous, and G\^{a}teaux-differentiable

\item For every $u\in X$ there is a number $c_{1}\left(  u\right)  $ such
that:
\[
\forall k\geq0,\ \ \ \left\Vert DF\left(  \varepsilon,x\right)  u\right\Vert
_{k}\leq c_{1}\left(  u\right)  \left(  m_{k}\left\Vert u\right\Vert
_{k+d_{1}}+\left\Vert F\left(  \varepsilon,x\right)  \right\Vert _{k}\right)
\]

\item There exists a linear map $L\left(  \varepsilon,x\right)  :Y\longrightarrow X$
such that:%
\[
DF\left(  \varepsilon,x\right)  L\left(  x\right)  =I_{Y}%
\]

\item For every $v\in Y$, we have:%
\[
\forall k\geq0,\ \ \ \left\Vert L\left(  \varepsilon,x\right)  v\right\Vert
_{k}\leq m_{k}^{\prime}\left\Vert v\right\Vert _{k+d_{2}}%
\]

\end{enumerate}

Then, for every $\varepsilon$ such that:
\[
\left\vert \varepsilon\right\vert < \min\left\{\frac{R}{m_{k_{0}}^{\prime}}\left\Vert
F_{1}\left(  0\right)  \right\Vert _{k_{0}+d_{2}}^{-1} \,,\,\varepsilon_0\right\} \text{ \ }%
\]
and every $\mu>m_{k_{0}}^{\prime}$, there is an $x_{\varepsilon}$ such that:%
\[
F\left(  \varepsilon,x_{\varepsilon}\right)  =0
\]
and:%
\[
\left\Vert x_{\varepsilon}\right\Vert _{k_{0}}\leq \mu \,\vert \varepsilon \vert\,\left\Vert
F_{1}\left(  0\right)  \right\Vert _{k_{0}+d_{2}}%
\]

\end{theorem}

\begin{proof}
Fix $\varepsilon$ with $\vert\varepsilon\vert<\varepsilon_0$. Consider the function:%
\[
G_{\varepsilon}\left(  x\right)  :=F_{0}\left(  x\right)  +\varepsilon\left(
F_{1}\left(  x\right)  -F_{1}\left(  0\right)  \right)
\]
It satisfies Conditions 1 to 5 of Theorem \ref{Thm10}. The equation $F\left(  \varepsilon
,x_{\varepsilon}\right)  =0$ can be rewritten $G_{\varepsilon}\left(
x\right)  =-\varepsilon F_{1}\left(  0\right)  :=y$. By Theorem \ref{Thm10},
we will be able to solve it provided:%
\[
\left\Vert y\right\Vert _{k_{0}+d_{2}}=\vert \varepsilon\vert\,\left\Vert F_{1}\left(
0\right)  \right\Vert _{k_{0}+d_{2}}<R\frac{1}{m_{k_{0}}^{\prime}}%
\]
and the solution $x_{\varepsilon}$ then satisfies
\[
\left\Vert x_{\varepsilon}\right\Vert _{k_{0}}\leq \mu\left\Vert y\right\Vert
_{k_{0}+d_{2}}=\mu \,\vert \varepsilon \vert\,\left\Vert F_{1}\left(  0\right)  \right\Vert
_{k_{0}+d_{2}}%
\]

\end{proof}

\section{Proof of Theorem \ref{thm2}}

The proof proceeds in three steps. Using Ekeland's variational principle, we
first associate with $\bar{y}\in Y$ a particular point $\bar{x}\in X$. From
then on, we argue by contradiction, assuming that $F\left(  \bar{x}\right)
\neq\bar{y}$. We then identify for every $n$ a particular direction
$u_{n}\in X$, and we investigate the derivative of
$x\rightarrow\sum\beta_{k}\left\Vert
F\left(  x\right)  -\bar{y}\right\Vert _{k}$ in the direction $u_{n}$. We
finally let $n\rightarrow\infty$ and derive a contradiction.

\subsection{Step 1}

We define a new sequence $\alpha_{k}$ by:%
\[
\alpha_{k}=\frac{\beta_{k+d_{2}}}{m_{k}^{\prime}}%
\]
and we endow $X$ with the distance $d$ defined by:%
\begin{equation}
d\left(  x_{1},x_{2}\right)  :=\sum_{k}\alpha_{k}\min\left\{  R,\left\Vert
x_{1}-x_{2}\right\Vert_{k}\right\}  \label{ek32}%
\end{equation}

By Proposition \ref{p2}, $\left(  X,d\right)  $ is a complete metric space.
Now consider the function $f:X\longrightarrow \mathbb{R}\cup\left\{  +\infty\right\}  $
(the value $+\infty$ is allowed) defined by:%
\begin{equation}
f\left(  x\right)  =\sum_{k=0}^{\infty}\beta_{k}\left\Vert F\left(  x\right)
-\bar{y}\right\Vert _{k} \label{100}%
\end{equation}

It is obviously bounded from below, and
\begin{equation}
0\leq\inf f\leq f\left(  0\right)  =\sum_{k=0}^{\infty}\beta_{k}\left\Vert
\bar{y}\right\Vert _{k}<\infty\label{101}%
\end{equation}

It is also lower semi-continuous. Indeed, let $x_{n}\longrightarrow x$ in
$\left(  X,d\right)  $. Then $x_{n}\longrightarrow x$ in every $X_{k}$. By
Fatou's lemma, we have:%
\[
\lim\inf_{n}~f\left(  x_{n}\right)  =\lim\inf_{n}\sum_{k}\beta_{k}\left\Vert
F\left(  x_{n}\right)  -\bar{y}\right\Vert _{k}\geq\sum_{k}\beta_{k}%
\lim\left\Vert F\left(  x_{n}\right)  -\bar{y}\right\Vert _{k}%
\]
and since $F:X\rightarrow Y$ is continuous, we get:%
\[
\sum_{k}\beta_{k}\lim\left\Vert F\left(  x_{n}\right)  -\bar{y}\right\Vert
_{k}=\sum_{k}\beta_{k}\left\Vert F\left(  x\right)  -\bar{y}\right\Vert
_{k}=f\left(  x\right)
\]
so that $\lim\inf_{n}~f\left(  x_{n}\right)  \geq f\left(  x\right)  $, as desired.

By assumption (\ref{y21}), we can take some $R^{\prime}<R$ with:%
\begin{equation}
\sum_{k=0}^{\infty}\beta_{k}\left\Vert \bar{y}\right\Vert _{k}<\frac
{\beta_{k_{0}+d_{2}}}{m_{k_{0}}^{\prime}}R^{\prime} \label{h}%
\end{equation}

We now apply Ekeland's variational principle to $f$ (see \cite{IE}%
,\ \cite{IE2}). We find a point $\bar{x}\in X$ such that:%
\begin{align}
f\left(  \bar{x}\right)   &  \leq f\left(  0\right) \label{y1}\\
d\left(  \bar{x},0\right)   &  \leq R^{\prime}\alpha_{k_{0}}\label{y2}\\
\forall x\in X,\ \ f\left(  x\right)   &  \geq f\left(  \bar{x}\right)
-\frac{f\left(  0\right)  }{R^{\prime}\alpha_{k_{0}}}d\left(  x,\bar
{x}\right)  \label{y3}%
\end{align}

Replace $f\left(  x\right)  $ by its definition (\ref{100}) in inequality
(\ref{y1}). \ We get:%
\begin{equation}
\sum\beta_{k}\left\Vert F\left(  \bar{x}\right)  -\bar{y}\right\Vert _{k}%
\leq\sum_{k=0}^{\infty}\beta_{k}\left\Vert \bar{y}\right\Vert _{k}%
<\infty\label{309}%
\end{equation}
and from the triangle inequality it follows that:%
\begin{equation}
\sum_{k=0}^{\infty}\beta_{k}\left\Vert F\left(  \bar{x}\right)  \right\Vert
_{k}\leq 2 \sum_{k=0}^{\infty}\beta_{k}\left\Vert \bar{y}\right\Vert _{k}<\infty\label{103}%
\end{equation}

If $\left\Vert \bar{x}\right\Vert _{k_{0}}>R^{\prime}$, then $d\left(  \bar
{x},0\right)  >R^{\prime}\alpha_{k_{0}}$, contradicting formula (\ref{y2}). So we must have
$\left\Vert \bar{x}\right\Vert _{k_{0}}\leq R^{\prime}<R$, and (\ref{t220}) is proved.

We now work on (\ref{y3}). To simplify notations, we set:%
\begin{equation}
A=\frac{\sum_{k=0}^{\infty}\beta_{k}\left\Vert \bar{y}\right\Vert _{k}%
}{R^{\prime}\alpha_{k_{0}}}=\frac{f\left(  0\right)  }{R^{\prime}\alpha
_{k_{0}}} \label{311}%
\end{equation}

It follows from (\ref{y3}) that, for every $u\in X$ and $t>0$, we have:%
\[
-\left(  f\left(  \bar{x}+tu\right)  -f\left(  \bar{x}\right)  \right)  \leq
Ad\left(  \bar{x}+tu,\bar{x}\right)
\]
and hence, dividing by $t$:%
\begin{equation}
-\frac{1}{t}\left[  \sum_{k=0}^{\infty}\beta_{k}\left\Vert \bar{y}-F\left(
\bar{x}+tu\right)  \right\Vert _{k}-\sum_{k=0}^{\infty}\beta_{k}\left\Vert
\bar{y}-F\left(  \bar{x}\right)  \right\Vert _{k}\right]  \leq\frac{A}{t}%
\sum_{k\geq0}\alpha_{k}\min\left\{  R,t\left\Vert u\right\Vert _{k}\right\}
\label{d5}%
\end{equation}

\subsection{Step 2.}

If $F\left(  \bar{x}\right)  =\bar{y}$, the proof is over. If not, we set:%
\begin{align*}
v  &  =F\left(  \bar{x}\right)  -\bar{y}\\
u  &  =-L\left(  \bar{x}\right)  v
\end{align*}
so that:
\[
DF\left(  \bar{x}\right)  u=-\left(  F\left(  \bar{x}\right)  -\bar{y}\right)
\]

Since $Y$ is standard, there is a sequence $v_{n}$ such that:%
\begin{align}
\forall k,\ \ \ \left\Vert v_{n}-v\right\Vert _{k}  &  \rightarrow
0\label{205}\\
\forall n,\ \ \left\Vert v_{n}\right\Vert _{k}  &  \leq c_{3}\left(  v\right)
\left\Vert v\right\Vert _{k}\label{206}\\
\left\Vert v_{n}\right\Vert _{k}  &  \leq c_{0}\left(  v_{n}\right)  ^{k}
\label{207}%
\end{align}

Set $u_{n}=-L\left(  \bar{x}\right)  v_{n}$. Clearly $\left\Vert
u_{n}-u\right\Vert _{k}\rightarrow0$ for every $k$. We have, using Condition
5:%
\begin{align}
\left\Vert u_{n}\right\Vert _{k}  &  \leq m_{k}^{\prime}\left\Vert
v_{n}\right\Vert _{k+d_{2}}\nonumber\\
&  \leq m_{k}^{\prime}c_{0}\left(  v_{n}\right)  ^{k+d_{2}} \label{300}%
\end{align}

We now substitute $u_{n}$ into formula (\ref{d5}), always under the assumption
that $F\left(  \bar{x}\right)  -\bar{y}\neq0$, and we let $t\longrightarrow0$.
If convergence holds, we get:%
\begin{equation}
-\lim_{\substack{t\longrightarrow0 \\t>0}}\frac{1}{t}\left[  \sum
_{k=0}^{\infty}\beta_{k}\left\Vert \bar{y}-F\left(  \bar{x}+tu_{n}\right)
\right\Vert _{k}-\sum_{k=0}^{\infty}\beta_{k}\left\Vert \bar{y}-F\left(
\bar{x}\right)  \right\Vert _{k}\right]  \leq A\lim_{
_{\substack{t\longrightarrow0 \\t>0}}}\sum_{k\geq0}\frac{\alpha_{k}}{t}%
\min\left\{  R,t\left\Vert u_{n}\right\Vert _{k}\right\}  \label{d50}%
\end{equation}

We shall treat the right- and the left-hand side separately, leaving $n$ fixed throughout.

We begin with the right-hand side. We have:%
\[
\frac{\alpha_{k}}{t}\min\left\{  R,t\left\Vert u_{n}\right\Vert _{k}\right\}
=\alpha_{k}\min\left\{  \frac{R}{t},\left\Vert u_{n}\right\Vert _{k}\right\}
=:\gamma_{k}\left(  t\right)
\]

We have $\gamma_{k}\left(  0\right)  \geq0$, $\gamma_{k}\left(  t^{\prime
}\right)  \geq\gamma_{k}\left(  t\right)  $ for $t^{\prime}\leq t$, and
$\gamma_{k}\left(  t\right)  \rightarrow\alpha_{k}\left\Vert u_{n}\right\Vert
_{k}$ when $t\rightarrow0$. By the Monotone Convergence Theorem:
\begin{equation}
\lim_{_{\substack{t\longrightarrow0 \\t>0}}}\sum_{k=0}^{\infty}\frac{1}%
{t}\alpha_{k}\min\left\{  R,t\left\Vert u_{n}\right\Vert _{k}\right\}
=\sum_{k=0}^{\infty}\alpha_{k}\left\Vert u_{n}\right\Vert _{k} \label{z16}%
\end{equation}

Now for the left-hand side of (\ref{d50}). Rewrite it as:%
\begin{equation}
-\sum_{k=0}^{\infty}\beta_{k}\frac{g_{k}\left(  t\right)  -g_{k}\left(
0\right)  }{t} \label{f3}%
\end{equation}
where:%
\[
g_{k}\left(  t\right)  :=\left\Vert \bar{y}-F\left(  \bar{x}+tu_{n}\right)
\right\Vert _{k}%
\]

We have $\left\Vert \bar{x}+tu_{n}\right\Vert _{k}\leq\left\Vert \bar
{x}\right\Vert _{k}+t\left\Vert u_{n}\right\Vert _{k}$. We have seen that
$\left\Vert \bar{x}\right\Vert _{k_{0}}\leq R^{\prime}<R$, so there is some
$\bar{t}>0$ so small that, for $0<t<\bar{t}$, we have $\left\Vert \bar
{x}+tu_{n}\right\Vert _{k_{0}}<R$. Without loss of generality, we can assume
$\bar{t}\leq1$. Since $F$ is G\^{a}teaux-differentiable, by Lemma \ref{l2}
$g_{k}$ has a right derivative everywhere, and:%
\begin{equation}
\vert \left(  g_{k}\right )_{+}^{\prime}(t)\vert=\lim_{\substack{h\rightarrow0
\\h>0}}\left\vert \frac{g_{k}\left(  t+h\right)  -g_{k}\left(  t\right)  }%
{h}\right\vert \leq\left\Vert DF\left(  \bar{x}+tu_{n}\right)  u_{n}%
\right\Vert _{k} \label{r1}%
\end{equation}

Introduce the function $f_{k}\left(  t\right)  =\left\Vert F\left(  \bar
{x}+tu_{n}\right)  \right\Vert _{k}$. It has a right derivative everywhere,
and $\left(  f_{k}\right)  _{+}^{\prime}\left(  t\right)  \leq\left\Vert
DF\left(  \bar{x}+tu_{n}\right)  u_{n}\right\Vert _{k}$, still by Lemma
\ref{l2}. We shall henceforth write the right derivatives $g_{k}^{\prime}$ and
$f_{k}^{\prime}$ instead of $\left(  g_{k}\right)  _{+}^{\prime}$ and $\left(
f_{k}\right)  _{+}^{\prime}$. By Condition 3, we have:
\begin{align*}
f_{k}^{\prime}\left(  t\right)   &  \leq c_{1}\left(  u_{n}\right)  \left(
m_{k}\left\Vert u_{n}\right\Vert _{k+d_{1}}+f_{k}\left(  t\right)  \right) \\
f_{k}^{\prime}\left(  t\right)  -c_{1}\left(  u_{n}\right)  f_{k}\left(
t\right)   &  \leq c_{1}\left(  u_{n}\right)  m_{k}\left\Vert u_{n}\right\Vert
_{k+d_{1}}%
\end{align*}

Integrating, we get:%
\begin{align*}
e^{-tc_{1}\left(  u_{n}\right)  }f_{k}\left(  t\right)  -f_{k}\left(
0\right)   &  \leq  \left(  1-e^{-tc_{1}\left(  u_{n}\right)  }\right)  m_{k}\left\Vert u_{n}%
\right\Vert _{k+d_{1}}\\
m_{k}\left\Vert u_{n}
\right\Vert _{k+d_{1}}+ f_{k}\left(  t\right)   &  \leq e^{tc_{1}\left(  u_{n}\right)  } \left[  m_{k}\left\Vert u_{n}\right\Vert
_{k+d_{1}}+ \left\Vert F\left(  \bar{x}\right)
\right\Vert _{k}\right] %
\end{align*}

Substituting this into Condition 3 and using (\ref{300}), we get:%
\begin{align*}
\left\Vert DF\left(  \bar{x}+tu_{n}\right)  u_{n}\right\Vert _{k}  &  \leq
c_{1}\left(  u_{n}\right)  \left(  m_{k}\left\Vert u_{n}\right\Vert _{k+d_{1}%
}+f_{k}\left(  t\right)  \right) \\
&  \leq c_{1}\left(  u_{n}\right) e^{tc_{1}\left(  u_{n}\right)  } \left[  m_{k}\left\Vert u_{n}\right\Vert
_{k+d_{1}}+ \left\Vert F\left(  \bar{x}\right)
\right\Vert _{k}\right] \\
&  \leq C_{1}\left(  u_{n}\right)  m_{k}m_{k+d_{1}}^{\prime}c_{0}\left(
v_{n}\right)  ^{k+d_{1}+d_{2}}+C_{1}\left(  u_{n}\right)  \left\Vert F\left(
\bar{x}\right)  \right\Vert _{k}=:\ell_{k}%
\end{align*}
where the term $C_{1}\left(  u_{n}\right):=c_{1}\left(  u_{n}\right) e^{c_{1}\left(  u_{n}\right)  }  $
depends on $u_{n}$ , but not on $k$. We have used the fact that $0<t<1$.

It follows from (\ref{r1}) that the function $g_{k}$ is $\ell_{k}%
$-Lipschitzian. So we get,
for every $k$ and $0<t<\bar{t}\,:$%
\[
\beta_{k}\left\vert \frac{g_{k}\left(  t\right)  -g_{k}\left(  0\right)  }%
{t}\right\vert \leq C_{1}\left(  u_{n}\right)  \beta_{k}m_{k}m_{k+d_{1}%
}^{\prime}c_{0}\left(  v_{n}\right)  ^{k+d_{1}+d_{2}}+C_{1}\left(
u_{n}\right)  \beta_{k}\left\Vert F\left(  \bar{x}\right)  \right\Vert _{k}%
\]

By assumption (\ref{y20}), the first term on the right-hand side
belongs to a convergent series. By inequality (\ref{103}), the second term is also summable.
So we can apply Lebesgue's Dominated Convergence Theorem to the series (\ref{f3}), yielding:%

\begin{equation}
\sum_{k=0}^{\infty}-\beta_{k}\,%
g'_{k}\left(  0\right) = \lim_{\substack{t\rightarrow0 \\t>0}}\sum_{k=0}^{\infty}-\beta_{k}\frac
{g_{k}\left(  t\right)  -g_{k}\left(  0\right)  }{t} \label{z17}%
\end{equation}

Writing (\ref{z16}) and (\ref{z17}) into formula (\ref{d5}) yields:%

\begin{equation}
\sum_{k=0}^{\infty} -\beta_{k}\,%
g'_{k}\left( 0\right) \leq A\sum_{k\geq
0}\alpha_{k}\left\Vert u_{n}\right\Vert _{k} \label{z18}%
\end{equation}

We now apply Lemma \ref{l2}.
Denote by $N_{k}$ the subdifferential of the norm in $Y_k$:%
\[
\forall  y\neq 0\,,\quad y^{\ast}\in N_{k}\left(  y\right)  \Longleftrightarrow\left\Vert y^{\ast
}\right\Vert _{k}^{\ast}=1\text{ and }\left\langle y^{\ast},y\right\rangle
_{k}=\left\Vert y\right\Vert _{k}\;.%
\]

There is some $y_{k}^{\ast}\left(  n\right)  \in N_{k}\left(  F\left(  \bar
{x}\right)  -\bar{y}\right)  $ such that:

\begin{equation}
g'_k(0)\,=\ <y_{k}^{\ast}\left(  n\right)  ,DF\left(
\bar{x}\right)  u_{n}>_{k} \label{q30}%
\end{equation}

Substituting into (\ref{z18}) we get:%
\begin{equation}
\sum_{k=0}^{\infty}-\beta_{k}<y_{k}^{\ast}\left(  n\right)  ,DF\left(
\bar{x}\right)  u_{n}>_{k}\ \leq A\sum_{k\geq0}\alpha_{k}\left\Vert u_{n}\right\Vert
_{k} \label{310}%
\end{equation}

Since $F\left(  \bar
{x}\right)  -\bar{y}\neq 0\,,$ the caracterization of $N_{k}\left(  F\left(  \bar
{x}\right)  -\bar{y}\right)  $ gives:
\begin{align}
<y_{k}^{\ast}\left(  n\right)  ,F\left(  \bar{x}\right)  -\bar{y}>_{k}  &
=\left\Vert F\left(  \bar{x}\right)  -\bar{y}\right\Vert _{k} \label{320}\\
\left\Vert y_{k}^{\ast}\left(  n\right)  \right\Vert _{k}^{\ast}  &
=1\label{321}%
\end{align}

\subsection{Step 3}

We now remember that $u_{n}=-L\left(  \bar{x}\right)  v_{n}$, so that
$DF\left(  \bar{x}\right)  u_{n}=-v_{n}$. Formula (\ref{310}) becomes:%
\begin{equation}
\sum_{k=0}^{\infty}\beta_{k}<y_{k}^{\ast}\left(  n\right)  ,v_{n}>_{k}\ \leq
A\sum_{k\geq0}\alpha_{k}\left\Vert L\left(  \bar{x}\right)  v_{n}\right\Vert
_{k} \label{330}%
\end{equation}

Set $\varphi_{k}\left(  n\right)  =\beta_{k}<y_{k}^{\ast}\left(  n\right)
,v_{n}>_{k}$ and $\psi_{k}\left(  n\right)  =\alpha_{k}\left\Vert L\left(
\bar{x}\right)  v_{n}\right\Vert _{k}$ .

On the one hand, by formula (\ref{206}), since $\left\Vert y_{k}^{\ast}\left(
n\right)  \right\Vert _{k}^{\ast}=1$, we have: $\ $%
\begin{align}
\left\vert \varphi_{k}\left(  n\right)  \right\vert  &  =\left\vert
\ \beta_{k}<y_{k}^{\ast}\left(  n\right)  ,v_{n}>_{k}\ \right\vert \leq
\beta_{k}\left\Vert y_{k}^{\ast}\left(  n\right)  \right\Vert _{k}^{\ast
}\left\Vert v_{n}\right\Vert _{k}\nonumber\\
&  \leq c_{3}\left(  v\right)  \beta_{k}\left\Vert v\right\Vert _{k}%
=c_{3}\left(  v\right)  \beta_{k}\left\Vert F\left(  \bar{x}\right)  -\bar
{y}\right\Vert _{k} \label{q31}%
\end{align}
and on the other, still by formula (\ref{206}), we have:%
\begin{align}
\left\vert \psi_{k}\left(  n\right)  \right\vert  &  =\alpha_{k}\left\Vert
L\left(  \bar{x}\right)  v_{n}\right\Vert _{k}=\frac{\beta_{k+d_{2}}}%
{m_{k}^{\prime}}\left\Vert L\left(  \bar{x}\right)  v_{n}\right\Vert
_{k}\nonumber\\
&  \leq\frac{\beta_{k+d_{2}}}{m_{k}^{\prime}}m_{k}^{\prime}\left\Vert
v_{n}\right\Vert _{k+d_{2}}=\beta_{k+d_{2}}\left\Vert v_{n}\right\Vert
_{k+d_{2}}\nonumber\\
&  \leq c_{3}\left(  v\right)  \beta_{k+d_{2}}\left\Vert v\right\Vert
_{k+d_{2}}=c_{3}\left(  v\right)  \beta_{k+d_{2}}\left\Vert F\left(  \bar
{x}\right)  -\bar{y}\right\Vert _{k+d_{2}} \label{q32}%
\end{align}

By inequality (\ref{309}), the series $\sum\beta_{k}\left\Vert F\left(  \bar
{x}\right)  -\bar{y}\right\Vert _{k}$ is convergent, so the last terms in
(\ref{q31}) and (\ref{q32}), which are independent of $n$,
form convergent series.

From (\ref{205}) and Condition 5, for each $k$, $\Vert L(\bar{x})\left(v_n-v\right)\Vert_k\to 0$ as $n\to \infty$.
We thus get the pointwise convergence of $\psi_k(n)=\alpha_{k}\left\Vert L\left(
\bar{x}\right)  v_{n}\right\Vert _{k}\,$:

\[
\lim_{n\to\infty}\psi_{k}\left(  n\right)  =\alpha_{k}\left\Vert L(\bar{x})v\right\Vert _{k}\,.\;
\]

Remembering that $v=F(\bar{x})-\bar{y}$, we have, by Formulas (\ref{320}) and (\ref{321}):
\[
\bigl\vert <y^*_k(n),v_n>_k-\Vert v \Vert_k \bigr\vert =
\left\vert <y^*_k(n),v_n-v>_k\right\vert\leq  \Vert v_n-v\Vert_k
\]
hence, from (\ref{205}), the pointwise convergence of $\varphi_k(n)=\beta_{k}<y^*_k(n),v_n>_k\,$:

\[\lim_{n\to\infty}\varphi_{k}\left(  n\right)  =\beta_k\Vert v \Vert_k\;.
\]

So, applying Lebesgue's Dominated Convergence Theorem to the series $\sum_k\varphi_{k}\left(  n\right)$ and $\sum_k \psi_{k}\left(  n\right)\,,$ we get:%
\begin{align*}
\lim_{n\rightarrow\infty}\ \sum_{k=0}^{\infty}\beta_{k}<y_{k}^{\ast}\left(
n\right)  ,v_{n}>_{k}  &  \ =\sum_{k=0}^{\infty}\beta_k\Vert v \Vert_k\\
\lim_{n\rightarrow\infty}\ \sum_{k=0}^{\infty}\alpha_{k}\left\Vert L\left(
\bar{x}\right)  v_{n}\right\Vert _{k}  &  \ =\sum_{k=0}^{\infty}\alpha_{k}\left\Vert L(\bar{x})v\right\Vert _{k}%
\end{align*}

It follows from the above and from (\ref{330}) that:%

\[
\sum_{k=0}^{\infty}\beta_{k}\left\Vert v\right\Vert _{k}\leq A\sum_{k\geq
0}\alpha_{k}\left\Vert L\left(  \bar{x}\right)  v\right\Vert _{k}=A\sum
_{k\geq0}\frac{\beta_{k+d_{2}}}{m_{k}^{\prime}}\left\Vert L\left(  \bar
{x}\right)  v\right\Vert _{k}\]

Estimating the right-hand side by Condition 5, we finally get:%
\[
\sum_{k=0}^{\infty}\beta_{k}\left\Vert v\right\Vert _{k}\leq A\sum_{k\geq
0}\beta_{k+d_{2}}\left\Vert v\right\Vert _{k+d_{2}}%
\]
with $v=F\left(  \bar{x}\right)  -\bar{y}\neq 0$, hence $A\geq1$.
Remembering the definition (\ref{311}) of $A$, this yields:
\[
\frac{\sum_{k=0}^{\infty}\beta_{k}\left\Vert \bar{y}\right\Vert _{k}%
}{R^{\prime}\alpha_{k_{0}}}=\frac{\sum_{k=0}^{\infty}\beta_{k}\left\Vert
\bar{y}\right\Vert _{k}}{R^{\prime}\beta_{k_{0}+d_{2}}}m_{k_{0}}^{\prime}%
\geq1
\]
which contradicts (\ref{h}). This shows that $F\left(  \bar{x}\right)
-\bar{y} $ cannot be non-zero, and concludes the proof.

\end{document}